\documentclass[10pt]{amsart}

\usepackage{amsfonts}
\usepackage{amsmath}
\usepackage{amssymb}
\usepackage{amsthm}
\usepackage{dynkin-diagrams}
\usepackage[alphabetic]{amsrefs}
\usepackage{calc}
\usepackage{graphicx}
\usepackage{hyperref}
\usepackage[left=1in,top=1in,right=1in,foot=1in]{geometry}
\usepackage{todonotes}
\usepackage{microtype}
\usepackage{tikz-cd}
\usepackage{tikz}
\usepackage{float}

\DeclareMathOperator{\spn}{span}

\newcommand{\sets}[2]{\left\{#1\,\middle|\,#2\right\}}
\newcommand{\genrel}[2]{\left\langle #1\,\middle|\,#2\right\rangle}

\newcommand{\stab}[2]{\mathrm{Stab}_{#1}(#2)}

\DeclareMathOperator{\Rep}{Rep}
\DeclareMathOperator{\Mat}{Mat}
\newcommand{\Bb}{\mathcal{B}}

\newcommand{\Z}{\mathbb{Z}}
\newcommand{\C}{\mathbb{C}}

\newcommand{\ssp}{\mathfrak{sp}}

\DeclareMathOperator{\Lie}{\mathrm{Lie}}

\newcommand{\GL}{\mathrm{GL}}
\newcommand{\PGL}{\mathrm{PGL}}
\newcommand{\SL}{\mathrm{SL}}

\newcommand{\SO}{\mathrm{SO}}
\newcommand{\Sp}{\mathrm{Sp}}

\newcommand{\Spin}{\mathrm{Spin}}
\newcommand{\Gm}{{\mathbb{G}_\mathrm{m}}}

\newcommand{\id}{\mathrm{id}}
\newcommand{\Spr}{\mathrm{Sp}_r}

\newcommand{\fm}{\mathfrak{m}}
\newcommand{\N}{\mathcal{N}}
\newcommand{\Nt}{\tilde{\mathcal{N}}}

\newtheorem{theorem}{Theorem}
\newtheorem{cor}{Corollary}
\newtheorem{prop}{Proposition}
\newtheorem{lem}{Lemma}

\newtheorem{conj}{Conjecture}

\theoremstyle{definition}
\newtheorem{dfn}{Definition}

\theoremstyle{remark}
\newtheorem{ex}{Example}
\newtheorem{rem}{Remark}

\newcommand{\F}{\mathcal{F}}
\newcommand{\Ce}{{\mathbb C}}
\newcommand{\Zet}{{\mathbb Z}}
\newcommand{\LG}{{ G\check{\ }}}
\newcommand{\LB}{{ B\check{\ }}}
\newcommand{\Lg}{{ \mathfrak g\check{\ }}}
\newcommand{\BB}{{\mathcal B}}
\renewcommand{\O}{{\mathcal O}}

\newcommand{\bbA}{{\mathbb A}}
\newcommand{\bP}{{\mathbb P}}

\newcommand{\gt}{{\tilde{\mathfrak g}\check{\ }}}

\newcommand{\Co}{{\bf Z}}

\makeatletter
\let\@wraptoccontribs\wraptoccontribs
\makeatother

\title{On the structure of the affine asymptotic Hecke algebras}

\author{Roman Bezrukavnikov}
\address{Department of Mathematics, MIT, Cambridge MA 02139 USA}
\email{bezrukav@math.mit.edu}

\author{Stefan Dawydiak}
\address{Max-Planck-Institut f\"ur Mathematik, 
Vivatsgasse 7, 53111 Bonn Germany}
\email{dawydiak@mpim-bonn.mpg.de}

\author{Galyna Dobrovolska}
\address{Ariel University, Ariel 40700, Israel}
\email{galdobr@gmail.com}

\address{Department of Mathematics, University of Toronto and Perimeter Institute
of Theoretical Physics, Waterloo, Ontario, Canada, N2L 2Y5
}
\email{braval@math.toronto.edu}

\address{The Hebrew University of Jerusalem
Givat Ram. Jerusalem 9190401, Israel} 
\email{kazhdan@math.huji.ac.il}

\contrib[With an appendix by]{Roman Bezrukavnikov, Alexander Braverman, and David Kazhdan}

\dedicatory{To the memory of James Humphreys}

\begin{document}
\begin{abstract} According to a conjecture of Lusztig, the asymptotic affine Hecke algebra should admit a description
in terms of the Grothedieck group of sheaves on the square of a finite set equivariant under the action of the 
centralizer of a nilpotent element in the reductive group. 
A weaker form of this statement, allowing for possible central extensions of stabilizers of that action, has been proved by 
the first named author with  Ostrik. In the present paper we describe an example showing
that nontrivial central extensions do arise, thus the above weaker statement is optimal. 

We also show that Lusztig's homomorphism from  the affine Hecke algebra
to the asymptotic
affine Hecke algebra induces an isomorphism on cocenters and discuss
the relation of the above central extensions to the structure of the cocenter.
\end{abstract}
\maketitle
\section{Introduction}
Let $G$ be a reductive algebraic group over $\Ce$. Let $W_f$ be its Weyl group and $W$ be the 
extended affine Weyl group, thus $W=W_f\ltimes \Lambda$ is the semi-direct product of $W_f$ with
the coweight lattice  $\Lambda$. Let 
$H$ be the 
affine Hecke algebra, 
thus $H$ is a deformation of the group algebra of $W$, i.e. $H$ is an algebra over $\Z[v,v^{-1}]$
such that $H/(v-1)=\Z[W]$.

The affine Hecke algebra $H$ plays a key role in several chapters of geometric representation theory,
such as representations of $p$-adic groups, affine Lie algebras and (geometric) Langlands duality.

Lusztig in his pioneering works \cite{cells2}, \cite{cells3} and \cite{cells4} has introduced a related algebra $J$ 
that can be thought of as a limit of $H$ as $v$ tends to $0$. In \cite{cells2} he introduced a homomomorphism
$\phi\colon H\to J\otimes_\Z\Z[v,v^{-1}]$ possessing a number of remarkable properties that provide a close link between the two algebras,
their representations etc. (see \cite{cells4}).

This link motivates a closer study of the algebra $J$. In \cite[Conjecture 10.5a)]{cells4} Lusztig conjectures a description of $J$
in terms of the Langlands dual group $\LG$.
More precisely, if $G$ is adjoint
 then for each nilpotent 
element in $\Lg=\Lie(\LG)$ 
the conjecture predicts existence of a finite set $Y_e$ with an action of the centralizer $Z_e$ of $e$ in $\LG$
so that
\begin{equation}\label{KZJ}
\bigoplus\limits_e K^{Z_e}(Y_e\times Y_e)\cong J_e,
\end{equation}
where the sum runs over a set of representatives of nilpotent conjugacy classes in $\Lg$ and 
$K^{Z_e}$ stands for the Grothendieck group of equivariant sheaves. 

In \cite{BO} a weaker form of this conjecture has been proved: namely, it established 
an isomorphism similar to \eqref{KZJ} where the left hand side is modified allowing for  central
extensions of the stabilizers   $\stab{Z_e}{y}$,   $y\in Y_e\times Y_e$. In the present paper we provide an example
showing that the stronger statement does not hold in general, i.e. that nontrivial central extensions have to be considered 
in order to account for the structure of a specific two-sided cell in the affine Weyl group of type $\tilde{B}_3$. 

We remark that the weak form of the conjecture also holds without the assumption that $G$ is adjoint, although it is quite easy to see
appearance of nontrivial central extension in that setting, the example is already provided by the lowest two-sided cell for $G=\SL(2)$,
$\LG=\PGL(2)$.  In the example considered in this paper we have $G=\SO(7)$, so $\LG=\Sp(6)$ is simply connected, however,
the centralizer of the nilpotent is not: the reductive part of its neutral connected component is isomorphic to $\PGL(2)$.

The idea to consider this particular example comes from the work of Losev and Panin \cite{LP} where a similar phenomenon is studied 
 in the setting of representations of finite W-algebras over a field of characteristic zero.

We present two independent arguments showing appearance of nontrivial central extensions. The first  one (see section
\ref{compB3}) is based on  an explicit computation of a summand in the algebra $J$
showing the structure of the based ring is inconsistent with an isomorphism (as based rings) $J_e\cong K^Z(Y\times Y)$.
A similar calculation has been accomplished recently by Yannan Qiu and Nanhua Xi  \cite{QX} (following
up on a question posed by the first named author); they have reached the same conclusion about existence of nontrivial central extensions. 
By considering the structure of the cocenter as a module over the center we show that $J_e$ is not isomorphic to $K^Z(Y\times Y)$ even as an 
abstract ring (see Corollary \ref{cor_Je}).

The second  argument (section \ref{section the springer fiber}) is based on geometry of a Springer fiber in the group $\LG=\Sp(6)$, which is linked to the structure of $J$ via
the noncommutative Springer resolution \cite{BeICM} by the results of \cite{B_two_real},  \cite{BL}. It shows that a specific realization
of $J_e$ in terms of a finite $Z_e$-set of centrally extended points does involve nontrivial central extensions. This approach links this phenomenon
to a geometric property of the corresponding Springer fiber, surjectivity of the map from the equivariant to the non-equivariant $K$-group.

In our example the group $Z_e$ turns out to be connected modulo the center of $\LG$.
 In this case  the conjecture  asserts\footnote{Here \label{footn} we use an additional assumption
 that the center of $\LG$ acts trivially on the set $Y$. We believe it is implicit in \cite[Conjecture 10.5a)]{cells4}, as it  follows from \cite[Conjecture 10.5c)]{cells4}. This is also a feature of the constructions
proving the weak form of the conjecture, see the proof of Lemma \ref{lem1}. However, Corollary \ref{cor_Je} excludes an isomorphism between $J_e$ and $K^Z(Y\times Y)$ even allowing for a nontrivial action of the center of $\LG$.}
that the corresponding summand $J_e$ of $J$ is isomorphic as a based algebra
to a matrix algebra $\Mat_n(\Rep(Z_e))$ where $\Rep(Z_e)$
is the representation ring of $Z_e$. 

\medskip

The second main theme of the paper is the structure of the cocenter $C=J/[J,J]=HH_0(J)$ of $J$. 
 Considering that cocenter is partly motivated by the 
 fact that by virtue of Lusztig's homomorphism $\phi:H\to J[v,v^{-1}]$ (or rather its specialization $\phi_q$ at $v=\sqrt{q}$ where $q$ is a power of a prime $p$)  a linear functional
 on the cocenter of $J$ defines an invariant distribution on the corresponding $p$-adic group.
 We prove (see Theorem \ref{cocenter_Thm}) that $\phi_q$ induces an isomorphism $\bar{\phi_q}$ on cocenters, thus the asymptotic Hecke algebra can be used to study invariant distributions. Surjectivity of the map
 $\bar{\phi_q}$ is proved by a direct computational argument; the proof of injectivity based on density of characters of standard representations appears in the Appendix by the first 
 named author, Braverman and Kazhdan. The appendix also contains two conjectures: Conjecture \ref{Phi} is concerned with extension of Theorem 
 \ref{cocenter_Thm} to the generalization of algebra $J$ and homomorphism $\phi_q$  defined in \cite{BrKa};
 Conjecture \ref{conj} (joint with Varshavsky) 
 gives an explicit description of the cocenter 
 in terms of the dual group. 
 
 The central extensions whose existence is demonstrated in the paper turn out to be closely related to 
the structure of the cocenter. 
 In particular, Conjecture \ref{conj} implies that appearance of nontrivial central extensions
 is a typical rather than an exceptional phenomenon, see Remark \ref{last_remark}. The explicit calculations done in the paper yield the proof of the conjecture
 in the special case under consideration.

\subsection{Acknowledgements}
It is an honor for us to dedicate this work to the memory of Jim Humphreys. His influence on the subject
both through his published work and through his generosity and enthusiasm in sharing his ideas can not
be overestimated.

We thank Jie Du, Do Kien Hoang, David Kazhdan, Ivan Losev, Victor Ostrik, Oron Propp, Yannan Qiu, Yakov Varshavsky, and Nanhua Xi for useful discussions.
We are also much indebted to the anonymous referee for pointing out a mistake in the original version of the paper.
R.B. was partly supported by the NSF and the Simons Foundation
sabbatical fellowship, S.D. was partly supported by NSERC. A.B. was partly supported by NSERC, and D.K. thanks ERC grant 669655.

\section{The main results}
\subsection{Set up and notations}
In this section we quickly recall the notation required to discuss our results in terms of the asymptotic
Hecke algebra.
\subsection{The affine Hecke algebra}
Let $W$ be as above and let $S\subset W$ be the set of simple reflections. Let $\ell$ be the length function
on $W$. Let $H$ 
be the affine Hecke algebra over $\Z[v,v^{-1}]$ attached to $W$. As a $\Z[v,v^{-1}]$-module, $H$ is free with standard basis $\{T_w\}_{w\in W}$. Multiplication is determined by the relations $T_wT_{w'}=T_{ww'}$ when $\ell(ww')=\ell(w)+\ell(w')$ and the quadratic relation $(T_s-v^2)(T_s+1)=0$ for $s\in S$. We write
$C_w$ for the Kazhdan-Lusztig basis element defined by
\[
C_w=\sum_{y\leq w}(-1)^{\ell(w)-\ell(y)}v^{\ell(w)}v^{-2\ell(y)}P_{y,w}(v^{-2})T_y,
\]
where the $P_{y,w}$ are the Kazhdan-Lusztig polynomials and $\leq$ is the strong Bruhat order. We write 
$h_{x,y,z}\in\Z[v,v^{-1}]$ for the structure constants of $H$ with respect to the $C_w$-basis.
\subsubsection{Cells, nilpotents, and the $a$-function}
\label{subsection cells nilpotents and the a-function}
Given an element $w$ in $W$, define the \emph{left descent set}
of $w$ by $\mathcal{L}(w)=\sets{s}{sx<x}\subset S$, and likewise the \emph{right descent set} $\mathcal{R}(w)$.

Using the Kazhdan-Lusztig polynomials, Lusztig defined partitions of $W$ into left, right, and two-sided cells.
For $w,x\in W$, write $x\prec w$ if $x< w$, $\ell(w)-\ell(x)$ is odd, and $P_{x,w}(v^2)$ has 
degree $\ell(w)-\ell(x)-1$. Write $x-w$ if either $x\prec w$ or $w\prec x$, and say that $x\leq_L w$ if there exists a sequence $x=x_0, x_1,x_2,\dots, x_n=w$ of elements of $W$ such that for each $i$, $1\leq i\leq n$, we have
$x_{i-1}-x_{i}$ and $\mathcal{L}(x_{i-1})\not\subset\mathcal{L}(x_{i})$. Say that $x\leq_{LR}w$ if there exists a sequence
$x=x_0, x_1, \dots, x_n=w$ of elements of $W$ such that for each $i$, $1\leq i\leq n$, either $x_{i-1} \leq_L  x_i$ or  $x_{i-1} ^{-1}\leq_L  x_i^{-1}$.
The \emph{left cells} of $W$ are the equivalence
classes under the equivalence relation associated to $\leq_L$; the two-sided cells are likewise defined from $\leq_{LR}$. The \emph{right cells} are defined to be of the form $\Gamma^{-1}$ where $\Gamma$ is a left cell. 
There are finitely-many cells of each type. Each two-sided cell is a union of one-sided cells.
Lusztig proved in \cite{cells4} that the two-sided cells of $W$ are in bijection with nilpotent
conjugacy classes in $\Lg$. 

Lusztig defined a function $a\colon W\to\N$ by setting $a(z)$ to be the smallest natural number such that $(-v)^{a(z)}h_{x,y,z}$ is a polynomial in $v$
for any $x,y\in W$. The $a$-function is constant on two-sided cells. As $W$ is crystallographic, $a(z)$ is always finite. In fact, if $e$ is the nilpotent corresponding to the cell
$c$ under Lusztig's bijection, then $a(c)=\dim_\C\Bb_e$, where $\Bb$ is the flag variety of $\Lg$ and
$\Bb_e$ is the Springer fiber corresponding to $e$.

Define $\gamma_{x,y,z}\in\Z$ to the constant term of $v^{a(z)}h_{x,y,z^{-1}}\in \Z[v]$.

Let $\delta(w)$ be the degree of $P_{1,w}(v)$, and define $\mathcal{D}=\sets{d\in W}{a(d)=l(d)-2\delta(d)}\subset W$.
The elements of $\mathcal{D}$ all obey $d^2=1$ and are called the \emph{distinguished involutions} of $W$.
There are finitely-many distinguished involutions in $W$. Each one-sided cell contains exactly one distinguished involution.
\subsection{The asymptotic Hecke algebra}
Let $J$ be the free abelian group with basis $\{t_w\}_{w\in W}$, equipped further with a ring structure where 
the product is given by
\[
t_xt_y=\sum_{z}\gamma_{x,y,z}t_{z^{-1}}.
\]
This multiplication is associative, and has identity element $1_J=\sum_{d\in\mathcal{D}}t_d$. 
If $c$ is 
a two-sided cell of $W$, denote 
\[
J_{c}=\spn{\sets{t_w}{w\in c}}=\left(\sum_{d\in\mathcal{D}\cap c}t_d\right)J\left(\sum_{d\in\mathcal{D}\cap c}t_d\right).
\]
The elements $t_d$ for $d\in\mathcal{D}$ are orthogonal idempotents, and
$J=\bigoplus_{c}J_c$ is a direct sum of two-sided ideals. The unit element in $J_{c}$
is $\sum_{d\in\mathcal{D}\cap c}t_d$.

If $\Gamma\subset c\subset W$ 
is a left cell, and $d_\Gamma$ is the unique distinguished involution in $\Gamma$, then
\[
J_{\Gamma\cap\Gamma^{-1}}=t_{d_\Gamma}Jt_{d_\Gamma}\subset J_c
\]
is a subring, with identity element $t_{d_\Gamma}$. All of the rings $J$, $J_c$, 
and $J_{\Gamma\cap\Gamma^{-1}}$ are based rings.
If $\Gamma$ and $\Gamma'$ are two left cells,
then $t_{d_\Gamma}Jt_{d_{\Gamma'}}=J_{\Gamma\cap(\Gamma')^{-1}}$ is a based 
$t_{d_\Gamma}Jt_{d_\Gamma}- t_{d_{\Gamma'}}Jt_{d_{\Gamma'}}$-bimodule.

From now on, we extend scalars to $\C$ and write $J$ for the asymptotic Hecke algebra with 
complex coefficients. We do the same for $J_c$ and $t_dJt_{d'}$.

\subsection{Equivariant sheaves and centrally extended points}
\label{section equivariant sheaves and centrally extended points}
We now recall from \cite{BO} the following 
\begin{dfn}
Let $Y$ be a finite set and $F$ be a reductive group. The structure of an $F$\emph{-set of centrally extended points} on $Y$ is the data of
\begin{enumerate}
\item[(a)] 
An $F$-action on $Y$;
\item[(b)]
For every $y\in Y$, a central extension
\begin{center}
 \begin{tikzcd}
 1\arrow[r]&\Gm\arrow[r]&\widetilde{\stab{F}{y}}\arrow[r]&\stab{F}{y}\arrow[r]&1,
 \end{tikzcd}
 \end{center} 
equivariant under the action of $F$ in the sense that for all $g\in F$ we are provided an isomorphism
$i_y^g$ such that
\begin{center}
 \begin{tikzcd}
 1\arrow[r]&\Gm\arrow[r]\arrow[d, "\id"]&\widetilde{\stab{F}{y}}\arrow[d, "i_y^g"]\arrow[r]&\stab{F}{y}\arrow[r]\arrow[d, "C_g"]&1\\
  1\arrow[r]&\Gm\arrow[r]&\widetilde{\stab{F}{gy}}\arrow[r]&\stab{F}{gy}\arrow[r]&1
 \end{tikzcd}
 \end{center} 
commutes, where $C_g$ is conjugation by $g$. We further require that $i^{g'g}_y=i^{g'}_{gy}\circ i^g_{y}$, 
and that if $g\in\stab{F}{y}$, we require that $i_y^g=C_g$.
\end{enumerate}
\end{dfn}

\begin{dfn}
Let $Y$ be a centrally-extended $F$-set. An $F$-\emph{equivariant sheaf on} $Y$ is the data of 
\begin{enumerate}
\item[(a)] 
A sheaf $\F$ of finite-dimensional $\C$-vector spaces on $Y$ with a projective $F$-equivariant structure;
\item[(b)]
For all $y\in Y$, an action of the central extension $\widetilde{\stab{F}{y}}$ on $\F_y$ such that
$\Gm$ acts by the identity character.
\end{enumerate}
\end{dfn}
If $Y$
has the structure of a centrally-extended $F$-set, then its square $Y\times Y$ has a canonical structure
of an $F$-extended set also, by defining $\widetilde{\stab{F}{(y,y')}}$ to be the product 
central extension $\tilde{F_y}\times_F\tilde{F}_{y'}/\Gm$, with $\Gm$ embedded antidiagonally and 
$\tilde{F_y}$ the restriction to $\stab{F}{(y,y')}$ of the central extension of $\stab{F}{y}$,
and likewise for $\tilde{F}_{y'}$.

If the central extension of $\stab{F}{y}$ is trivial, then the 
$\widetilde{\stab{F}{y}}$-action on $\F_y$ is just the data of a $\stab{F}{y}$-action.

\subsection{Criteria for existence of nontrivial central extensions}

We now present two criteria for appearance of nontrivial central extensions.

In the next two statements we fix a nilpotent $e \in \Lg$ and let $c_e$ be the corresponding cell in $W$ and $J_e$ 
the corresponding summand in $J$. Let $Z_e$ be the  centralizer of $e$ and $Z_e^{red}$ be a maximal reductive subgroup in $Z_e$.

According to \cite[Theorem 4]{BO}, the based ring $J_e$ is isomorphic to $K^{Z_e^{red}}(Y_e\times Y_e)$ for a certain centrally extended
$Z_e^{red}$-set $Y_e$. 

Furthermore, the set $Y_e$ can be made more explicit based on a result of \cite{BL}. Recall that the {\em noncommutative Springer resolution} $A$
\cite{BeICM}, \cite{BeMi} is a certain noncommutative ring equipped with a derived equivalence 
\begin{equation}\label{der_eq}
D^b(A-mod)\cong D^b(Coh(\Nt)),
\end{equation}
 where $A-mod$ 
stands for the category of finitely generated modules and $\Nt=T^*(\LG/\LB)$ is the Springer resolution. The ring $\O(\N)$ of regular functions
on the nilpotent cone $\N\subset \Lg$ is identified with the center of $A$; for $e\in \N$ let $A_e=A/{\fm}_e A$ denote the specialization of $A$ at the corresponding
maximal ideal $\fm_e\subset \O(\N)$. 

Then 
 \cite[Proposition 8.25]{BL} provides an isomorphism of based rings  
$$J_e\cong K^{Z_e^{red}}(Y_e\times Y_e)$$ where $Y_e$ is the set of isomorphism
classes of irreducible $A_e$ modules, equipped with the natural action of $Z_e^{red}$: 
it is deduced from the isomorphism of based rings
\begin{equation}\label{bimod}
 J_e \cong K^{Z_e^{red}}(A_e-bimod^{ss}),
 \end{equation}
 where the right hand side is
  the $K$-group of the monoidal category of semisimple $A_e$-bimodules 
equivariant under $Z_e^{red}$.

\begin{lem}\label{lem1}
Suppose that the reductive part of the centralizer $Z_e^{red}$ of $e$ is connected modulo the center of $\LG$. 

a) Let $\epsilon$ be an indecomposable idempotent in 
$J_e$. Then $\epsilon J \epsilon\cong Rep(Z_e^{red})$.

b) Let $\epsilon_1$, $\epsilon_2$ be indecomposable idempotents in 
$J_e$. If $\epsilon_1 J \epsilon_2$ is not isomorphic to a regular
$\epsilon_1 J \epsilon_1-\epsilon_2 J \epsilon_2$ bimodule, then the corresponding
point in $Y_e^2$ is centrally extended. \qed
\end{lem}

\proof It is easy to see that the center of $\LG$ acts trivially on the set $Y_e$:
in view of the above interpretation of $Y_e$ as the set of irreducible $A_e$-modules,
this follows from the fact that the action of $\LG$ on $A_e$ factors through the adjoint group. The statement is now clear from the
definition of equivariant sheaves on the square of a centrally extended set. \qed

Let  $\BB_e$ denote the corresponding Springer fiber, i.e. the preimage of $e$ under the Springer map $\Nt\to \N$. 

\begin{lem}\label{lem2}
Suppose the reductive part $Z_e^{red}$ of the centralizer is connected modulo the center of $\LG$. If the central extensions are trivial then the restriction of equivariance map
$K^{Z_e^{red}}(\BB_e)=K^{Z_e^{red}}(A_e-mod)\to K(\BB_e)$ is surjective.
\end{lem}

\proof
We again use \eqref{bimod}.
Since $Z_e^{red}$ is assumed to be connected modulo the center of $\LG$, every irreducible  $A_e$-module admits a compatible projective action (i.e. an action of a
central extension) of $Z_e^{red}$. Clearly if there exist two irreducible modules $L$, $L'$ of $A_e$, so that  $L$ is $Z_e^{red}$-equivariant
and $L'$ is not (thus, $L'$ is a genuinely projective representation of $Z_e^{red}$) then the bimodule $L'\otimes L^{op}$ is also a genuinely projective
representation of $Z_e^{red}$, thus the corresponding point in $Y_e^2$ is centrally extended.

It is easy to see that the only irreducible module $L$ such that the corresponding object $\F_L$ in $D^b(Coh_{\BB_e}(\gt))$ satisfies $R\Gamma(\F_L)\ne 0$
is $Z_e^{red}$ equivariant. This follows from the fact that its projective cover corresponds to the structure sheaf of the formal neighborhood of $\BB_e$
under the equivalence \eqref{der_eq}, while the structure sheaf is clearly equivariant.

It remains to show that if all irreducible $A_e$-modules are $Z_e^{red}$-equivariant then the map $K^{Z_e^{red}}(\BB_e)\to K(\BB_e)$ is surjective. The derived equivalence \eqref{der_eq}
yields compatible isomorphisms $K(Coh(\BB_e))\cong K(A_e-mod)$, $K(Coh^{Z_e^{red}}(\BB_e))\cong K^{Z_e^{red}}(A_e-mod)$, see \cite{BeMi}. 
Thus we are reduced to showing that the map $K^{Z_e^{red}}(A_e-mod)\to K(A_e-mod)$
is not surjective.

Suppose that an irreducible $A_e$-module $L$ carries a  genuinely projective $Z_e^{red}$-action and let $n>1$ be the order of
the corresponding class in $H^2(Z_e^{red},\Ce^*)$. Then for a virtual representation $\sum n_i [L_i]$ in the image of the map $K^{Z_e^{red}}(A_e-mod)\to K(A_e-mod)$
the coefficient at $[L]$ is divisible by $n$, so the map is not surjective.
\qed
\subsection{Cocenters}

Let $C=J/[J,J]$, $C_e=J_e/[J_e,J_e]$ be the cocenter (zeroth Hochschild homology) of $J$, $J_e$ respectively, thus $C=\bigoplus C_e$ where $e$ runs over the set
of conjugacy classes of nilpotents in $\Lg$.
As before, $H_q$ denotes the specialization of the affine Hecke algebra at $q\in \C^\times$ and $\phi_q:H_q\to J$
is the specialization of Lusztig's homomorphism $\phi$ defined in \cite{cells2}.

\begin{theorem}\label{cocenter_Thm}
If $q$ is not a root of unity then the homomorphism $\phi_q:H_q\to J$ induces an isomorphism $\bar{\phi_q}:H/[H,H]\to C$.
\end{theorem}

Surjectivity of the map $\bar{\phi_q}$ is checked in section \ref{sur_coc}. The proof of injectivity appears in the Appendix.


 Theorem \ref{cocenter_Thm} provides a close link between algebra $J$ and 
 harmonic analysis on the $p$-adic group
$G_F$ where $F$ is a non-Archimedian local field. 
 Namely, let $S(G_F)$ be the space of locally constant complex valued functions with compact support on $G_F$ and let $S_0(G_F)$ be
 the $G_F^2$ submodule generated by the space $S(I\backslash G_F/ I)$ of functions bi-invariant under the Iwahori subgroup.
  It is well known that $S(I \backslash G_F/I)\cong H_q$.
 Thus we 
 get a map from the cocenter  $C(H_q)=H_q/[H_q,H_q]$ of $H_q$ to the space of coinvariants of the conjugation 
 action $S_0(G_F)_{G_F}$ which is well known to be an isomorphism;
  also    $S_0(G_F)$
 splits off canonically as a direct summand in $S(G_F)$. 
 Thus the linear dual $C^*$ of  $C$ is realized as a direct summand in the space of invariant distributions on $G_F$. 
 
 Notice that Theorem \ref{cocenter_Thm} implies that this summand further splits as a direct sum $\oplus _e C(J_e)$ indexed
 by unipotent conjugacy classes in $\LG$; these appear in harmonic analysis on the $p$-adic group as a part
 of the Langands parameter. A more precise conjectural  description of that space in terms of $\LG$ is presented in the Appendix.

\section{Computations in $J$ for type $\tilde{B}_3$}\label{compB3}
\subsection{The two-sided cell $D$}
From now on we specialize $W=W(\tilde{B}_3)$ to be the affine Weyl group of type $\tilde{B}_3$.
The cell structure of $W$ is completely understood thanks to work of J. Du, and we shall refer to \cite{DuCell}
and the exposition in the survey \cite{ShiSurv}. We shall adopt the union of the notation
used in these references. 

We now fix a presentation of $W$ that we shall use throughout.  Fix
\[
W=\genrel{s_0,s_1,s_2,s_3}{s_i^2=(s_1s_2)^3=(s_0s_2)^3=(s_2s_3)^4=1,~(s_is_j)^2=1~\text{otherwise}},
\]
corresponding to the labelling
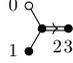
\begin{center}
\dynkin[labels={0,1,2,3}] B[1]{3}
\end{center}
of the affine Dynkin diagram. Here $s_0$ is the affine simple reflection. We write $W_f$
for the finite Weyl group of type $B_3$.

We recall now some results of \cite{DuCell} as exposited in \cite{ShiSurv}.
There are eight two-sided cells of $W$, written
\[
A=C(9),~B=C(6),~C=C(4),~D=C(3),~E=C(2)_2,~F=C(2)_1,~ G=C(1),~H=C(0)=\{1\},
 \]
where the value of Lusztig's $a$-function on $C(n)$ is $n$.

Under the bijection recalled in section \ref{subsection cells nilpotents and the a-function} (first proven in this case in \cite{DuCell}) the nilpotent $e\in\ssp_6(\C)$ having three equal Jordan blocks corresponds to the two-sided cell $D$, and we have $Z_e^{red}\cong\SO(3)\times\{\pm 1\}$.

We shall require combinatorial descriptions of the sets $C'(n)$ and $C(n)$ from \cite{DuCell} for our calculations, and we recall them now. Let $S(i)=\sets{I\subset S}{\ell(w_I)=i}$, $S(2)_1=\{\{s_0,s_1\}\}$ and $S(2)_2=S(2)\setminus S(2)_1$.
Here $w_I$ is longest element in the subgroup generated by $I$.
Then the sets $C'(i)$ for $i\neq 2$ are defined to be the sets of all $w\in W$ such that
\begin{enumerate}
\item[(a)]
$w=xw_Iz$ for some $x,z\in W$, and $I\in S(i)$ with $\ell(w)=\ell(x)+\ell(w_I)+\ell(z)$, and
\item[(b)]
$w\neq xw_Jz$ for any $x,z\in W$ and $J\not\subset\bigcup_{j=0}^{i}S(j)$ 
with $\ell(xw_Jz)=\ell(x)+\ell(w_J)+\ell(z)$. 
\end{enumerate}
The condition (b) ensures that one does not double-count elements that will satisfy 
criterion (a) for two-sided cells with higher $a$-values. If $i=2$, then the conditions defining $C'(2)_j$ are the same, except only $S(2)_j$ appears
in (b).

Then one puts $C(n)=C'(n)$ for $n=1,6,9$, $C(2)_j=C'(2)_j$, 
$C(3)=C'(3)\setminus C''(3)$, and $C(4)=C'(4)\cup C''(3)$, where
\[
C''(3)=\sets{xyz}{\ell(xyz)=\ell(x)+\ell(y)+\ell(z),~ y\in\{s_1s_2s_1s_3s_2s_1, s_0s_2s_0s_3s_2s_0\}}.
\]
It will be relevant to note explicitly that
\[
S(3)=\sets{I\subset S}{\ell(w_I)=3}=\{\{s_1,s_2\},\{s_0,s_2\},\{s_0,s_1,s_3\}\}.
\]
We have
\[
w_{\{0,1,3\}}=s_0s_1s_3,~~~w_{\{1,2\}}=s_1s_2s_1,~~~w_{\{0,2\}}=s_0s_2s_0.
\]
Moreover, we have
\[
S(4)=\{\{s_2,s_3\}\},~~S(6)=\{\{s_0,s_1,s_2\}\},~~S(9)=\{\{s_1,s_2,s_3\},\{s_0,s_2,s_3\}\}.
\]
This gives $3+1+1+2$ longest words, and accounting for four singleton subsets and the three
subsets of the commuting generators $s_0,s_1,s_3$ giving longest words of length $2$, we see that
no others subsets $S(n)$ can occur. We have
\[
w_{\{1,2,3\}}=s_3s_2s_3s_1s_2s_3s_1s_2s_1.
\]
As conditions (a) and (b) are simple to check via computer, this description affords an easy way to determine 
whether an element of $W$ belongs to $D$. We shall omit these checks, which we carried out in Sage 
\cite{sagemath}, from now on.
\subsection{One-sided cells in $D$}
Du showed that $D$ contains twelve right (and hence twelve left) cells. Thus $D$ contains twelve 
distinguished involutions, which we list in the table following Theorem \ref{partition theorem} below.

We shall write $\Gamma_{X}$ for the right cell in $D$ the elements of which all have
left descent set equal to $X$. Right cells in $D$ are totally determined by 
their left descent sets, except when this set is equal to $\{2\}$; following \cite{DuCell}, 
we shall distinguish these three right cells via the decorations $\{2\}$, $\{2'\}$ and $\{\hat{2}'\}$. The same 
is then of course
true for left cells in $D$, and we shall write the left cell with right descent set $X$ as $\Gamma_X^{-1}$.
\subsubsection{Right spectra and right primitive pairs}
Given an element $w\in W$, define the \emph{right spectrum} $\Spr(w)$ of $w$ to be the set
\[
\sets{y\in W}{\text{there is a sequence}~y=x_0,x_1,\dots, x_n=w~\text{s.t.}~ x_i=x_{i-1}s_{j(i)}~\text{and}~\mathcal{R}(x_{i-1})\not\subset\mathcal{R}(x_i)\not\subset\mathcal{R}(x_{i-1})}.
\]
Define $\Spr(w)\leq\Spr(w')$ as explained in Section 3.4. of \cite{DuCell}. If $\Spr(w)$
is equivalent to $\Spr(w')$ under the corresponding equivalence relation, we say that $w$ 
and $w'$ are an $r$-\emph{primitive pair}. Elements of these equivalence classes belong 
to the same right cell.
\subsubsection{Graph of the $r$-primitive pair}
In \cite{DuCell}, Du computed several graphs elucidating the cell structure for $W$. The graph relevant to 
$D$ is
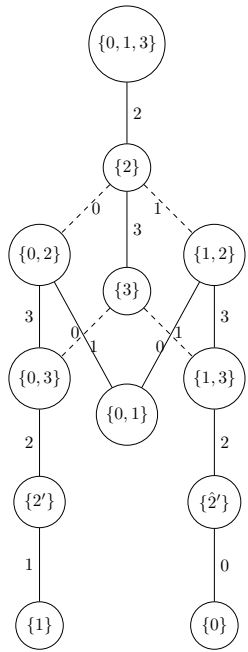
\begin{figure}[H]
\centering
\resizebox{0.2\textwidth}{!}{
\begin{tikzpicture}[ node distance={25mm}, main/.style = {draw, circle}] 
\node[main] (1) {$\{0,1,3\}$}; 
\node[main] (2) [below of=1] {$\{2\}$};
\node[main] (3) [below of=2] {$\{3\}$};
\node       (4) [below of=3] {};
\node[main] (5) [below of=3] {$\{0,1\}$};
\node[main] (6) [below right of=2] {$\{1,2\}$};
\node[main] (7) [below left of=2] {$\{0,2\}$};
\node[main] (8) [below of=6] {$\{1,3\}$};
\node[main] (9) [below of=7] {$\{0,3\}$};
\node[main] (10) [below of=8] {$\{\hat{2}'\}$};
\node[main] (11) [below of=9] {$\{2'\}$};
\node[main] (12) [below of=10] {$\{0\}$};
\node[main] (13) [below of=11] {$\{1\}$};
\draw (1)-- node[anchor=west] {2} (2) ;
\draw (2)--node[anchor=west] {3} (3);
\draw (6) -- node[anchor=west] {3}(8);
\draw (8) --node[anchor=west] {2}(10);
\draw (10)--node[anchor=west] {0}(12);
\draw (7)--node[anchor=east] {3} (9) ;
\draw (9)-- node[anchor=east] {2} (11)  ;
\draw (11)--node[anchor=east] {1}(13) ;
\draw (7)-- node[anchor= north west] {1}(5);
\draw (6) --node[anchor= north east] {0}(5);
\draw[dashed] (3)-- node[anchor=east] {0} (9);
\draw[dashed] (3)--node[anchor= west] {1}(8);
\draw[dashed] (2) --node[anchor= west] {0}(7);
\draw[dashed] (2)--node[anchor= east] {1}(6);
\end{tikzpicture}
}
\end{figure}
It is the graph for the $r$-primitive pair $d_{\{0,1,3\}}$ and $d_{\{0,1,3\}}s_2s_1$.

We use the version of the graph defined for right spectra; Du uses left spectra and left primitive pairs (thus 
in the notation of \cite{DuCell}, p.1395, for us $d_{12}=s_0s_1s_3s_2s_1$  and not $s_1s_2s_0s_1s_3$).

The graph is constructed as follows. Starting with the elements $s_0s_1s_3$ and $s_0s_1s_3s_2s_1$,
draw the graphs of $\Spr(s_0s_1s_3)$ and $\Spr(s_0s_1s_3s_2s_1)$ by drawing solid edges between vertices $x$ and $y$ in the same spectrum if $x<y$, $y=xs$ for some $s\in S$, and $\mathcal{R}(x)\neq \mathcal{R}(y)$. Between a vertex 
labelled by $x\in\Spr(s_0s_1s_3)$ and a vertex labelled by $xs_i\in\Spr(s_0s_1s_3s_2s_1)$, draw 
a dotted line if $xs_i\not\in\Spr(s_0s_1s_3)$. 
\begin{lem}[\cite{DuCell}]
The set of elements appearing as vertices in the graph gives a set of representatives of the left cells 
in $D$. All elements appearing as vertices belong to $\Gamma_{\{0,1,3\}}$.
\end{lem}
The graph therefore gives a cross-section of $D$ along the right cell $\Gamma_{\{0,1,3\}}$.
\subsection{The partition}
\subsubsection{Coxeter group automorphisms and based ring automorphisms}
Write 
\[
R:=\Rep(\PGL(2)\times\{\pm 1\})\cong\Rep(\PGL(2))[x]/(x^2-1),
\]
and $M$ for the Grothendieck group of $\SL(2)\times\{\pm 1\}$-representations with odd highest weight.
The ring $R$ is a based ring with basis $\{V(2n),xV(2n)\}_{n\geq 0}$, and both $R$ and $M$
are based $R-R$-bimodules, the latter with basis $\{V(2n+1), xV(2n+1)\}_{n\geq 0}$.

Note that $W$ has Coxeter group automorphism exchanging $s_0$ and $s_1$, that we denote $\varphi$. 
The automorphism $\varphi$ then sends left cells to left cells, 
right cells to right cells, and two-sided cells to two-sided cells. Because
it is easy to see that it fixes some elements in $D$, we have that $D$
is $\varphi$-stable. Moreover, we have \cite{LX}
\[
\gamma_{x,y,z}=\gamma_{\varphi(x),\varphi(y),\varphi(z)}
\]
for all $x,y,z\in W$, and so $\varphi$ induces a based ring automorphism of $J$. When $d$ is a distinguished involution in $D$ such that
$\varphi(d)=d$, $\varphi$ therefore induces a based ring automorphism of $t_dJt_d$.
The only such morphism is the identity, and so $\varphi$ fixes the corresponding 
rings pointwise. The automorphism $\varphi$ acts on the graph above by 
reflecting about the obvious fixed central column. 
\subsubsection{The partition}
Recall from section \ref{section equivariant sheaves and centrally extended points} that 
$J_D\cong K^{\PGL_2(\C)\times\{\pm 1\}}(Y\times Y)$, where
$Y$ is a centrally-extended $\PGL_2(\C)\times\{\pm1\}$-set of size $12$; one may take the underlying set of $Y$
to be $D\cap\mathcal{D}$ for ease of indexing. If the central extensions do
not appear in the isomorphism of \cite{BO}, then $J_D\cong\Mat_{12}(R)$. We will show that this fails: for some $y\in Y$,
the associated Schur multiplier $s_y$ for the stabilizer of $y$ is nontrivial, \textit{i.e.} some point 
$y$ is centrally extended. This will also follow independently from the geometric 
proof in section \ref{section the springer fiber}. Moreover, we will now determine precisely when this 
happens for the product set $Y\times Y$.

We may realize the ring $J_D$
as having elements given by $12\times 12$ matrices with diagonal entries in $R$, and off-diagonal entries
either in $R$ or in $M$, this last case occurring at entries indexed by distinguished involutions $d,d'$
such that exactly one of the Schur multipliers $s_d,s_{d'}$ is nontrivial. The monomial matrices
with nonzero entry at $d,d'$ make up the $t_dJt_d-t_{d'}Jt_{d'}$-bimodules $t_dJt_{d'}$. The diagonal
monomial matrices make up the rings $t_dJt_d$.

By Lemma \ref{lem1} (a), we have that for any $d\in\mathcal{D}\cap D$,
the ring $t_dJt_d$ is isomorphic as a based ring to $R$, and there exists a partition
$\mathcal{D}_A\sqcup\mathcal{D}_B=\mathcal{D}\cap D$ such that if both $d,d'\in\mathcal{D}_?$,
then under the isomorphism $t_dJt_d\cong R\cong t_{d'}Jt_{d'}$ of based rings, 
$t_dJt_{d'}\cong R$ as a based $R-R$-bimodule. On the other hand, if
$d\in \mathcal{D}_A$ but $d'\in\mathcal{D}_B$, then $t_dJt_{d'}$ is isomorphic to $M$ as 
a based $R-R$-bimodule.
\begin{theorem}
\label{partition theorem}
The partition mentioned above is
\[
\mathcal{D}_A=\{d_{\{0,1,3\}}, d_{\{2\}},d_{\{3\}}\},~~\mathcal{D}_B=\{d_{\{0,2\}}, d_{\{0,1\}}, d_{\{1,2\}}, d_{\{0,3\}}, d_{\{1,3\}},d_{\{2'\}}, d_{\{\hat{2}'\}}, d_{\{1\}}, d_{\{0\}}\}.
\]
That is, the sets of the partition are exactly the graphs of $\Spr(d_{\{0,1,3\}})$
and of $\Spr(d_{\{0,1,3\}}s_2s_1)$.
\end{theorem}
The theorem is summarized by the table
\begin{center}
\begin{tabular}{|c|c|c|c|}
\hline 
$\mathcal{L}(c)$ & Distinguished involution & Orbit & Component\\ 
\hline 
0,1,3 & $s_0s_1s_3$ & 1 & $\Spr(d_{\{0,1,3\}})$   \\ 
\hline 
2 & $s_2s_0s_1s_3s_2$ & 2 &$\Spr(d_{\{0,1,3\}})$  \\ 
\hline 
0,2 & $s_0s_2s_0$  & 3 & $\Spr(d_{\{0,1,3\}}s_2s_1)$ \\ 
\hline 
1,2 & $s_1s_2s_1$ & 3 &$\Spr(d_{\{0,1,3\}}s_2s_1)$  \\ 
\hline 
0,1 & $s_1s_2s_0s_2s_1$ & 4 & $\Spr(d_{\{0,1,3\}}s_2s_1)$\\ 
\hline 
1,3 & $s_1s_3s_2s_1s_3$ & 5& $\Spr(d_{\{0,1,3\}}s_2s_1)$ \\ 
\hline 
3 & $s_3s_2s_0s_1s_3s_2s_3$ & 6&$\Spr(d_{\{0,1,3\}})$  \\ 
\hline 
0,3 & $s_3s_0s_2s_0s_3$ & 5& $\Spr(d_{\{0,1,3\}}s_2s_1)$ \\ 
\hline 
2' & $s_2s_3s_0s_2s_0s_3s_2$ & 7& $\Spr(d_{\{0,1,3\}}s_2s_1)$ \\ 
\hline 
1 & $s_1s_2s_3s_0s_2s_0s_3s_2s_1$ & 8&$\Spr(d_{\{0,1,3\}}s_2s_1)$  \\ 
\hline 
$\hat{2}'$ & $s_2s_1s_3s_2s_1s_3s_2$ & 7 & $\Spr(d_{\{0,1,3\}}s_2s_1)$  \\ 
\hline 
0 & $s_0s_2s_1s_3s_2s_1s_3s_2s_0$ & 8 & $\Spr(d_{\{0,1,3\}}s_2s_1)$   \\ 
\hline 
\end{tabular} 
\end{center}
where the rows are indexed by right cells of $D$; for a right cell $\Gamma$, 
the ``orbit" column describes which $\varphi$-orbit $\Gamma$ belongs to
and $\mathcal{L}(c)$ is the set of indices of simple reflections corresponding to 
$\mathcal{L}(w)\subset S$ for any $w\in c$. 
\begin{rem}
The other two-sided cells of $W$ with multiple spectra appearing in the graphs of an $r$-primitive
pair are $A,B,C$. For example, it is immediate that multiple spectra cannot appear in the graphs 
of an $r$-primitive pair for the subregular cell $G$, by uniqueness of reduced expressions for 
elements of the subregular cell. We note that it is known in many cases that the subregular
summand of $J$ is a matrix algebra, by work of Xu \cite{XuSub}.
\end{rem}
\begin{lem}
The distinguished involutions in $D$ are precisely those appearing in the table in Theorem
\ref{partition theorem}.
\end{lem}
\begin{proof}
By counting, it suffices to show that every element in the table is a distinguished involution. 
For each element $w$ in the table, we can check that $\ell(w)-2\deg P_{e,w}=3$ by 
appealing to the calculations carried out by Goresky in \cite{GoreskyTable}. To conclude, it 
suffices to check that each element in the table lies in the cell $C(3)$, the unique
two-sided cell in $W$ on which $a=3$. 

By the inequality $a(w)\leq \ell(w)-2\deg P_{e,w}$ and
theorem 7.6 (b) of \cite{ShiSurv}, we see that any element in the table not in $C(3)$
is contained in either $C(2)_2$, $C(2)_1$, $C(1)$, or $C(0)$. As each element of the table is by 
construction of form $w=xw_Iy$ with $\ell(w)=\ell(x)+\ell(w_I)+\ell(y)$ and $\ell(w_I)=3$,
each element fails condition (b) (respectively (b)') required to lie in 
$C(1)$ (respectively $C(2)_2$ or $C(2)_1$). Clearly no element in the table lies in $C(0)$.
Therefore every element in the table is in $C(3)$ and the lemma is proved.
\end{proof}
Therefore all of the subrings $t_dJt_d\subset J_D$ and bimodules $t_dJt_{d'}$ are given by choosing
distinguished distinguished involutions from the table above.

\subsection{Proof of Theorem \ref{partition theorem}}
We will prove Theorem \ref{partition theorem} by computing judiciously chosen products in 
$t_dJt_d\cdot t_dJt_{d'}$ for 
various distinguished involutions $d,d'$ and comparing the results against the Clebsch-Gordon rule.
Namely, we note that if 
\[
V(2n)\otimes V(m)=V(2n+m)\oplus V(\lambda)
\]
is a sum of two irreducible representations,
then $\lambda=2n-m$, $2n\geq m$, and $2n+m-(2n-m)=2$ whence $m=1$. On the other hand, if $n\neq 0$ and
\[
V(2n)\otimes V(m)=V(\lambda),
\]
then $m=0$ and $\lambda=2n$. Of course, the same pattern in the number of summands still holds when taking into account the character $x$ of $\{\pm 1\}$. Therefore to show that $d$ and $d'$ are in the same set, it suffices to produce
elements $t_x\in t_dJt_d$ and $t_y\in t_dJt_{d'}$ such that $x\neq d$ and $t_xt_y=t_z$ for some 
$z$. To show that $d$ and $d'$ are in different sets, it suffices to produce elements
$t_x$ and $t_y$ as above such that $t_xt_y=t_z+t_{z'}$ for $z\neq z'$. Moreover, it 
suffices to test just one element in each $\varphi$-orbit. This reduces the number of necessary computations to 
seven. Upon producing an element in $y\in\Gamma_X\cap\Gamma_Y^{-1}$, we shall write it $y=y^Y_X$.
\subsubsection{$d_{\{0,1,3\}}$ and $d_{\{0,2\}}$ lie in different sets}
First we will show that $d_{\{0,1,3\}}$ and $d_{\{0,2\}}$ are in different sets.

Using the description of $D$ in \cite{DuCell} and descent sets, we calculate that
\[
x_{\{0,1,3\}}=d_{\{0,1,3\}}s_2d_{\{0,1,3\}}\in\Gamma_{\{0,1,3\}}\cap\Gamma_{\{0,1,3\}}^{-1},
\]
and by construction
\[
y_{\{0,1,3\}}^{\{0,2\}}:=d_{\{0,1,3\}}s_2s_0\in\Gamma_{\{0,1,3\}}\cap\Gamma_{\{0,2\}}^{-1}.
\]
Next, we compute that
\[
C_{x_{\{0,1,3\}}}C_{y_{\{0,1,3\}}^{\{0,2\}}}=
-(v^{-3}+3v^{-1}+3v+v^3)\left(C_{s_3s_1s_0s_2s_3s_1s_0s_2s_0} + C_{s_3s_2s_1s_0s_2s_1s_0} + C_{s_3s_1s_0s_2s_0}\right).
\]
The element $s_3s_2s_1s_0s_2s_1s_0$ does not lie in $D$. Indeed, we see that 
$s_3s_2s_1s_0s_2s_1s_0=x\cdot w_{A_3,0}$
is a length-increasing multiplication with the longest word in the parabolic subgroup of 
$W$ corresponding to $A_3$. It follows from \cite{ShiSurv} that $a(s_3s_2s_1s_0s_2s_1s_0)>3$.
On the other hand, we compute that $s_3s_1s_0s_2s_3s_1s_0s_2s_0$ is in $D$. Therefore we have
\[
t_{x_{\{0,1,3\}}}t_{y_{\{0,1,3\}}^{\{0,2\}}}=
t_{s_3s_1s_0s_2s_3s_1s_0s_2s_0} + t_{y_{\{0,1,3\}}^{\{0,2\}}}.
\]
Comparing this to the Clebsch-Gordon rule, we see that $t_{y_{\{0,1,3\}}^{\{0,2\}}}$
must correspond to $V(1)$, and $x_{\{0,1,3\}}$ to $V(2)$. It follows that $d_{\{0,1,3\}}$ and $d_{\{0,2\}}$
are in different sets, and hence that $d_{\{0,1,3\}}$ and $d_{\{1,2\}}=\varphi(d_{\{0,2\}})$ are in different 
sets. Denote the set containing $d_{\{0,1,3\}}$ by $\mathcal{D}_A$, and the set containing $d_{\{0,2\}}$
and $d_{\{1,2\}}$ by $\mathcal{D}_B$.
\subsubsection{$d_{\{2\}}$ lies in $\mathcal{D}_A$}
Next we claim that $d_{\{2\}}$ lies in $\mathcal{D}_A$. Note that $y_{\{0,1,3\}}^{\{2\}}=d_{\{0,1,3\}}s_2$
is in $D$, by direct computation; its left descent set shows that it is in $\Gamma_{d_{\{0,1,3\}}}$. 
By construction of the graph in \cite{DuCell}, $d_{d_{\{0,1,3\}}}s_2\in\Gamma_{d_2}^{-1}$. We compute that
\[
C_{x_{\{0,1,3\}}}C_{y_{\{0,1,3\}}^{\{2\}}}=-(v^{-3}+3v^{-1}+3v+v^3)C_{d_{\{0,1,3\}}s_2d_{\{0,1,3\}}s_2},
\]
we see that $t_{y_{\{0,1,3\}}^{\{2\}}}$ must correspond to $V(0)$ and that $d_{\{0,1,3\}}$ 
and $d_{\{2\}}$ are in the same set, so that $d_{\{2\}}\in\mathcal{D}_A$.
\subsubsection{$d_{\{0,1\}}$ lies in $\mathcal{D}_B$}
\label{unconditional counterexample}
We claim that $d_{\{0,1\}}$ lies in $\mathcal{D}_B$. The distinguished involution $d_{\{0,1\}}$
is the only $\varphi$-fixed distinguished involution in $\mathcal{D}_B$.
The element $y^{\{0,1\}}_{\{0,1,3\}}=d_{\{0,1,3\}}s_2s_0s_1$ is in 
$\Gamma_{\{0,1,3\}}\cap\Gamma_{\{0,1\}}^{-1}$, again by construction. We compute that
\begin{multline*}
C_{x_{\{0,1,3\}}}C_{y^{\{0,1\}}_{\{0,1,3\}}}=
-(v^{-3}+3v^{-1}+3v+v^3)\left(C_{s_3s_1s_0s_2s_3s_1s_0s_2s_1s_0} +C_{s_3s_1s_0s_2s_1s_0}\right)
\\ 
+(v^{-4}+4v^{-2}+6+4v^2+v^4)C_{s_3s_2s_1s_0s_2s_1s_0}
\end{multline*}
The element $s_3s_2s_1s_0s_2s_1s_0$ is not in $D$, but
$s_3s_1s_0s_2s_3s_1s_0s_2s_1s_0$ and $s_3s_1s_0s_2s_1s_0$ are. As $J_D$ is a two-sided ideal, we have
\[
t_{x_{\{0,1,3\}}}t_{y^{\{0,1\}}_{\{0,1,3\}}}=t_{s_3s_1s_0s_2s_3s_1s_0s_2s_1s_0} +t_{s_3s_1s_0s_2s_1s_0}.
\]
We see that $d_{\{0,1\}}\in\mathcal{D}_B$.

\subsubsection{$d_{\{1,3\}}$ and $d_{\{0,3\}}$ lie in $\mathcal{D}_B$}
We claim that $d_{\{0,3\}}$ and $d_{\{1,3\}}$ lie in $\mathcal{D}_B$.
The element $y_{\{0,2\}}^{\{0,3\}}=d_{\{0,2\}}s_3$ is in $D$, because it is a length-increasing
product starting from $d_{\{0,2\}}$, and is too short to be excluded. Its descent sets show that
it must be in $\Gamma_{\{0,2\}}\cap\Gamma_{\{0,3\}}^{-1}$. For the sake of variety, let $x_{\{0,2\}}=d_{\{0,2\}}d_{\{0,1,3\}}d_{\{0,2\}}$. We see computationally that $x_{\{0,2\}}\in\Gamma_{\{0,2\}}\cap\Gamma_{\{0,2\}}^{-1}$.  We therefore compute
\[
C_{x_{\{0,2\}}}C_{y_{\{0,2\}}^{\{0,3\}}}=
-(v^{-3}+2v^{-1}+2v+v^3)C_{s_0s_2s_3s_1s_0s_2s_0s_3}.
\]
We conclude that $d_{\{0,2\}}$ and $d_{\{0,3\}}$ are in the 
same set. Hence $d_{\{0,3\}},d_{\{1,3\}}\in\mathcal{D}_B$.
\subsubsection{$d_{\{3\}}$ lies in $\mathcal{D}_A$} 
We claim that $d_{\{3\}}$ lies in $\mathcal{D}_A$.
Put $y_{\{0,1,3\}}^{\{3\}}=d_{\{0,1,3\}}s_2s_3$. By construction $y_{\{0,1,3\}}^{\{3\}}\in \Gamma_{\{0,1,3\}}\cap\Gamma_{\{3\}}^{-1}$. We compute
\[
C_{x_{\{0,1,3\}}}C_{y_{\{0,1,3\}}^{\{3\}}}=
-(v^{-3}+3v^{-1}+3v+v^3)C{s_3s_1s_0s_2s_3s_1s_0s_2s_3}.
\]
We conclude according to the reasoning of the previous cases that $d_{\{3\}}\in\mathcal{D}_A$.
\subsubsection{$d_{\{2'\}}$ and $d_{\{\hat{2}'\}}$ lie in $\mathcal{D}_B$}
We claim that $d_{\{0,1,3\}}$ and $d_{\{2'\}}$ are in different sets.
The element $y_{\{0,1,3\}}^{\{2'\}}=d_{\{0,1,3\}}s_2s_0s_3s_2=s_1s_3d_{\{0,2\}}s_3s_2$
lies in $\Gamma_{\{0,1,3\}}\cap\Gamma_{\{2'\}}^{-1}$ by construction.
We compute
\begin{equation}
\label{2', 2'hat computation equation}
C_{x_{\{0,1,3\}}}C_{y_{\{0, 1, 3\}}^{\{2'\}}}=
-(v^{-3}+3v^{-1}+3v+v^3)\left(C_{s_3s_1s_0s_2s_3s_1s_0s_2s_3s_0s_2}+C_{s_3s_2s_1s_0s_2s_3s_1s_0s_2}
+C_{s_3s_1s_0s_2s_3s_0s_2}\right).
\end{equation}
Again we compute that the middle term is not in $D$, but the longest and shortest
terms are. As before we compare with the Clebsch-Gordon rule and determine that this
module structure corresponds to $M$. Hence $d_{\{2'\}},d_{\{\hat{2}'\}}\in\mathcal{D}_B$.
\subsubsection{$d_{\{1\}}$ and $d_{\{0\}}$ lie in $\mathcal{D}_B$}
We claim that $d_{\{0,1,3\}}$ and $d_{\{1\}}$ are in different sets. Define 
$y_{\{0, 1, 3\}}^{\{1\}}:=y_{\{0, 1, 3\}}^{\{2'\}}s_1$, so
that $y_{\{0, 1, 3\}}^{\{1\}}\in\Gamma_{d_{\{0,1,3\}}}\cap\Gamma_{d_{\{1\}}}^{-1}$ as in all the previous cases. 
We have
\[
C_{y_{\{0, 1, 3\}}^{\{1\}}}=C_{y_{\{0, 1, 3\}}^{\{2'\}}}C_{s_1},
\]
and thus can reuse some of our calculations from the previous case, multiplying each term in the right hand side 
of \eqref{2', 2'hat computation equation}. We have
\[
C_{s_3s_1s_0s_2s_3s_0s_2}C_{s_1}=C_{s_3s_1s_0s_2s_3s_0s_2s_1}
\]
and $s_3s_1s_0s_2s_3s_0s_2s_1\in D$. Recall that $s_3s_2s_1s_0s_2s_3s_1s_0s_2\not\in D$
and in fact one has $a(s_3s_2s_1s_0s_2s_3s_1s_0s_2)>3$. We have
\[
C_{s_3s_2s_1s_0s_2s_3s_1s_0s_2}C_{s_1}=\sum_{z\leq_R s_3s_2s_1s_0s_2s_3s_1s_0s_2}\tilde{\mu}(z,s_3s_2s_1s_0s_2s_3s_1s_0s_2) C_z,
\]
where the coefficients $\tilde{\mu}$ are all integers. This follows from 
(\cite{LusUnequal}, Theorem 6.6) because 
\[
s_1\not\in\mathcal{R}(s_3s_2s_1s_0s_2s_3s_1s_0s_2).
\]
Now, $z\leq_R s_3s_2s_1s_0s_2s_3s_1s_0s_2$ by definition implies $z\leq_{LR}s_3s_2s_1s_0s_2s_3s_1s_0s_2$, thus 
by \cite{LusUnequal}, P4, we have 
\[
a(z)\geq a(s_3s_2s_1s_0s_2s_3s_1s_0s_2)>3,
\]
and so $C_{s_3s_2s_1s_0s_2s_3s_1s_0s_2}C_{s_1}$ makes no contribution. Finally, we have
\[
C_{s_3s_1s_0s_2s_3s_1s_0s_2s_3s_0s_2}C_{s_1}=C_{s_3s_1s_0s_2s_3s_1s_0s_2s_3s_0s_2s_1},
\]
and $s_3s_1s_0s_2s_3s_1s_0s_2s_3s_0s_2s_1$ is in $D$. It follows that
\[
t_{x_{\{0,1,3\}}}t_{y_{\{0, 1, 3\}}^{\{1\}}}=
t_{s_3s_1s_0s_2s_3s_1s_0s_2s_3s_0s_2s_1}+t_{s_3s_1s_0s_2s_3s_0s_2s_1}
\]
and so $d_{\{1\}},d_{\{0\}}\in\mathcal{D}_B$. 
\qed
\subsection{Computation of the cocenter}\label{cocSO6}
We record the following consequence of the above computation. Let $R$, $\tilde R$ be the rings of invariant functions on 
$\PGL(2)$ and $\SL(2)$ respectively. Let $A$ be the  subring in $\Mat_2(\tilde R)$ consisting of matrices
$(a_{ij})$ such that $a_{ij}(-x)=(-1)^{i-j}a_{ij}(x)$. 

\begin{cor}\label{cor_Je}
a) The algebra $J_e$ is Morita equivalent to $A\otimes \Ce[\Zet/2\Zet]$.

b) The center of  $J_e$ is isomorphic $R\otimes \Ce[\Zet/2\Zet]$. 
The cocenter $J_e/[J_e,J_e]$  is isomorphic, as a module over the center, to $(R\oplus k)\otimes \Ce[\Zet/2\Zet]$, where
$R$ acts on $k$ via the evaluation at the order two element $s\in \PGL(2)$.
The isomorphism sends a matrix $A=(a_{ij})$ to $(Tr(A), a_{11}(\tilde s))$ where $\tilde s$ is an order 4 element in $\SL(2)$.  

c) The abstract (rather than based) ring $J_e$ is not isomorphic to a ring of the form $K^{Z_e}(Y\times Y)$ for any finite $Z_e$-set $Y$.
\end{cor}

\proof Part a) follows directly from the fact that both subsets in the partition of Theorem \ref{partition theorem} are nonempty.
 Part b) easily follows from a). 
Since $Z_e\cong \PGL(2) \times \Zet/2\Zet$, for any finite $Z_e$-set $Y$ we have
$K^{Z_e}(Y\times Y)\cong K^{\Zet/2\Zet}(Y\times Y)\otimes R$. Since  $K^{\Zet/2\Zet}(Y\times Y)$ is easily seen to be
a semisimple ring, it follows that the cocenter of $K^{Z_e}(Y\times Y)$ is a free module over its center. Statement (b) shows that this property is not
shared by the ring $J_e$.  \qed
\section{The Springer fiber}
\label{section the springer fiber}
In this section we show existence of nontrivial central extensions by an argument based on geometry of the Springer fiber $\BB_e$
using the criterion of Lemma \ref{lem2}. 

Let   $V$ be a 6 dimensional symplectic space, we have $\LG=\Sp(V)$.

We write $V=V_2\otimes V_3$ where $V_2$ is a two-dimensional symplectic and $V_3$ is a 3 dimensional orthogonal 
one.
Let $\bar{e}$ be a nonzero nilpotent acting on $V_2$ and $e=\bar{e}\otimes Id_{V_3}$.
Then the reductive part of the centralizer $Z=\mathrm{O}(3) =\SO(3)\times \{\pm 1\}$; in particular, it is connected
modulo the center of $\LG$. 

Consider the Springer fiber $\BB_e$.

\begin{prop}
 The natural map 
\begin{equation}\label{1}
K^Z(\BB_e)\to K(\BB_e)
\end{equation} is not surjective.
\end{prop}

\proof
The variety  $\BB_e$ has 3 components $X_1$, $X_2$, $X_3$.

Here $X_1$ consists of isotropic flags $(F_1\subset \cdots \subset F_6)$ such that
$F_3=Im(e)=Ker(e)$, thus $X_1\cong SL(3)/B_{SL(3)}$.

The component $X_2$ consists of flags such that $F_2\subset Im(e)=Ker(e)$ and it satisfies
the following condition. Notice that the projectivization  $\bP(Ker(e))$ is identified with $\bP(V_3)$,
so it contains a fixed quadric $Q$; we require that the line $\bP(F_2)$ is tangent to $Q$.

The component $X_3$ consists of flags such that $\bP(F_1)\in Q$. 

\medskip

Let $U$ be  the open subset in $X_2$ consisting of flags such that $\bP(F_1)\not \in Q$ 
and $F_3\ne Ker(e)$. 

Then $U\cap (X_1\cup X_3)=\emptyset$. It follows that the map $K(\BB_e)\to K(U)$ is onto,
so if  \eqref{1} was surjective then the map $K^Z(U)\to K(U)$ would also be surjective.

However we have a $Z$-equivariant map $U\to Q$ with fiber $\bbA^2$. 
We have $K^Z(U)= K^Z(Q)$, $K(U)=K(Q)$, while the map $K^Z(Q)\to K(Q)$ is not onto,
since the class of a vector bundles of odd degree  of $Q\cong \bP^1$ does not equal
 the class of a sheaf equivariant under $Z\cong \PGL(2)\times \{ \pm 1\}$. \qed

\section{Cocenters}%

\subsection{Surjectivity of $\bar{\phi_q}$}  \label{sur_coc}
The goal of this subsection is the proof of the following
\begin{prop}
\label{prop bar phi_q is onto}
Suppose that $q$ is not a root of unity. The map $\bar{\phi_q}$ is onto.
\end{prop}

We fix a 2-sided cell $c$ corresponding to  the conjugacy class of a 
unipotent element $e$. Let $H_c$ be the corresponding cell subquotient of $H$. The map $\phi_q$ restricts to a map 
$\phi_c\colon H_c\to J_c$. Define $\psi\colon H_c\to J_c$
by $\psi(C_w)=t_w$, where $w$ is in $c$.

\begin{lem}
\label{lemIdeals}
Let $A$ be an associative, unital ring and $I$ be a left ideal not annihilating any simple $A$-module.
Then the map $I\to HH_0(A)$ is surjective.
\end{lem}
\begin{proof}[Proof of Lemma \ref{lemIdeals}]
Let $I$ be a left ideal not annihilating any simple $A$-module,
and let $\mathcal{I}$ be the two-sided ideal generated by $I$. Then
\[
\mathcal{I}=\sets{\sum i_ja_j}{a_j\in A, i_j\in I}.
\]
Since $i_ja_j=a_ji_j \mod [A,A]$, and $a_ji_j\in I$, we see that
the images of $\mathcal{I}$ and $I$  in $HH_0(A)$ coincide. Now, if $\mathcal{I}$ does 
not annihilate any simple $A$-module, then $\mathcal{I}=A$, this follows from
the standard fact that any maximal two-sided ideal is primitive. 
\end{proof}
\begin{proof}[Proof of Proposition \ref{prop bar phi_q is onto}]
Consider the filtration by the two-sided cell ideals on $H$ and on $J$.
The homomorphism $\phi$ is compatible with these filtrations, thus
it suffices to show surjectivity of the map of associated graded spaces
for the induced filtration on $HH_0$. Thus it is enough 
 to check that the image of $\phi_c$ maps surjectively 
to the cocenter of $J_c$. 

Using \cite{cells2}, 2.4 (d), we see that:
$$\psi(C_{x_1}C_{x_2})t_{x_3}=\psi(C_{x_1}\psi^{-1}(t_{x_2}t_{x_3}))$$
for $x_1,\, x_2, \, x_3\in c$. 
Applying the antiautomorphisms of $H$ and $J$ given by
 $C_w\mapsto C_{w^{-1}}$ and $t_w\mapsto t_{w^{-1}}$ we get:
 $$t_{x_1} 
 \psi(C_{x_2}C_{x_3})=\psi( \psi^{-1}(t_{x_1}t_{x_2}) C_{x_3}).$$

Fix $x_1, x_2$ and let $x_3$ vary over the set of distinguished involutions of $c$;
summing the resulting expressions we get:
$$t_{x_1}\phi_c(C_{x_2}) = \phi_c(\psi^{-1} (t_{x_1}t_{x_2})).$$
This shows that $\phi_c(H_c)$ is a left ideal in $J_c$. By 
\cite[Proposition 4.4]{cells4},  $\phi_c(H_c)$ does not annihilate 
any simple $J_c$-module. The proposition now follows from 
Lemma \ref{lemIdeals}.
\end{proof}
\begin{rem}
As we see from the proof of the proposition and \cite{Xi}, surjectivity of $\bar{\phi_q}$ holds in fact 
whenever $q$ is not a root of the Poincar\'{e} polynomial of $W_f$.
\end{rem}

\subsection{Discussion of Conjecture \ref{conj}}
We now present some examples and remarks pertaining to  Conjecture \ref{conj} presented in the Appendix.
\begin{ex}
 Suppose that $Z_e$ is connected and simply connected. Then it is not hard to see that
every function satisfying condition (1) in the statement of the conjecture is pulled back under the projection to the first factor, thus
 $\bar{\O}_e \cong \O(Z_e^{red})^{Z_e^{red}}$ is the ring of conjugation invariant functions on $Z_e^{red}$.
 This agrees with the fact that in this case $J_e$ is isomorphic to a matrix algebra over $\Rep(Z_e^{red})$.
 \end{ex}
 \begin{ex}
 \label{our-example}
  Suppose that $G=\Spin(7)$, $\LG=\mathrm{PSp}(6)$ and $e$ are as in the previous section, thus $Z_e^{red}= \PGL(2)$ (we have
  passed to the adjoint version of $\LG$ to simplify notation). Then it is easy to see that $\bar{\O}_e$ is the direct sum of two subspaces:
$pr_1^*\left(\O(Z_e^{red})^{Z_e^{red}}\right)$ (functions pulled back from the first factor) and the one dimensional 
space of functions supported on the image of the set $\{(x,y)\in \PGL(2) \times \GL(2)\ |\ xy=-yx\}$. This agrees with the computation of $C_e$
in subsection \ref{cocSO6}.
\end{ex}
\begin{rem}\label{last_remark}
The Conjecture provides an effective criterion for appearance of nontrivial central
extensions. For example, suppose that $Z_e^{red}$ is connected modulo the center of $\LG$ but its derived group
is not simply connected. Fix an isomorphism $Z_e^{red}=H/K$ where $H$ is a connected
reductive group and $K$ is a finite subgroup in its center.

Suppose that one can find $x,y\in H$ so that $xyx^{-1}y^{-1}$ is a nontrivial element of $K$
and the image of $x$ in $Z_e^{red}$ is regular in $G$.

Then it is easy to see that the map $\O (Z_e^{red})^{Z_e^{red}}\to \bar{O}_e$ is not onto,
which implies appearance of a nontrivial central extension.
\end{rem} 

\begin{rem}
It is interesting to compare Conjecture \ref{conj} with the conjecture at the end of \cite{QX}.
The latter asserts an isomorphism 
\begin{equation}\label{conj_QX}
J_e\cong K^{Z_e}((\BB_e^{\Ce^*})^2)
\end{equation}
where $\BB_e^{\Ce^*}$ is the space of the fixed points of the multiplicative group $\Ce^*$ acting on the Springer fiber
via the homomorphism $SL(2)\to \LG$ coming from an $sl(2)$ triple $(e,h,f)$. 

Assuming \eqref{conj_QX} one can define a map  $J_e\to K(Coh^{Z_e}(Z_e))$, $[\F]\mapsto [pr_*a^*\F]$ where
$pr:Z_e\times \BB_e^{\Ce^*}\to Z_e$ is the projection and $a:Z_e\times \BB_e^{\Ce^*}\to (\BB_e^{\Ce^*})^2$
is given by $(z,x)\mapsto (x,z(x))$. It is not hard to check that it factors through a map $C_e\to K(Coh^{Z_e}(Z_e))$.

One can compose it with the natural map $K(Coh^{Z_e}(Z_e))\to \O_e$
(where we use notations of section \ref{conj_descr})
sending $[\F]$ to $f_{\F}:(g,\gamma)\mapsto Tr(g,\F_\gamma)$ where $\F_\gamma$ denotes the (derived)
fiber of $\F$ at $\gamma$. It is easy to see that this map lands in the space of functions satisfying condition (1)
of Conjecture \ref{conj}. One expects that the composed map lands in $\bar{\O_e}$ and yields the isomorphism 
of Conjecture \ref{conj}. 

\end{rem}

\begin{rem}
Let us point out another relation between conjectural isomorphism \eqref{conj_QX} and presence of nontrivial central
extensions in the description of $J_e$. We illustrate this connection imposing a simplifying assumption that $Z_e$ is connected
modulo the center of $\LG$. We also assume that the center of $\LG$ acts trivially on the set $Y$ (see footnote \ref{footn} in the Introduction).

Then in the absence of central extensions the specialization 
of $J_e$ at any central character would be isomorphic to the matrix algebra $\Mat_n(\Ce)$, $n=
\dim H^*(\BB_e)$. Assuming \eqref{conj_QX}, we see that the center of $J_e$ is identified 
with $K(\Rep(Z_e))$, so that a central character corresponds to
a semisimple conjugacy class $s\in Z_e$. The specialization $K^{Z_e}((\BB_e^{\Ce^*})^2)
\otimes_{K(\Rep(Z_e))} \Ce_s$ maps naturally to $K(Coh(\BB_{e,s}^{\Ce^*})^2)\otimes \Ce$,
where $\BB_{e,s}^{\Ce^*}$ is the space of simultaneous fixed points of $e$, $s$ and $\Ce^*$,
one can check that this map is compatible with the convolution product.

The image of this map is clearly contained in the invariants of the centralizer 
of $s$ in $Z_e$
acting on $K(Coh(\BB_{e,s}^{\Ce^*})^2)\otimes \Ce$.
It is not hard to deduce from the results of \cite[Chapter 8]{CG} that dimension of  
$K(Coh(\BB_{e,s}^{\Ce^*})^2)\otimes \Ce$ equals $n^2$, thus the rank of the algebra homomorphism
$$K^{Z_e}((\BB_e^{\Ce^*})^2)
\otimes_{K(\Rep(Z_e))} \Ce_s \to K(Coh(\BB_{e,s}^{\Ce^*})^2)\otimes \Ce$$
is less than $n^2$ provided that the action of the centralizer on $K(Coh(\BB_{e,s}^{\Ce^*})^2)$
is nontrivial. This implies that the corresponding specialization of $J_e$ is not isomorphic to
$\Mat_n(\Ce)$, so nontrivial central extensions have to appear.

In particular, for the example considered in the present work (see Example \ref{our-example})
one lets $s$ be the only order two conjugacy class in $Z_e$, then one can check that $\BB_{e,s}^{\Ce^*}$ 
is isomorphic to the disjoint union of three projective lines and 6 isolated points. The centralizer 
of $s$ in $Z_e$ has two connected components, an element in the nontrivial component acts by
a nontrivial permutation (namely, a product of three commuting transpositions) of the six isolated points.
This yields another proof of appearance of nontrivial central extensions in this case, conditional on 
\eqref{conj_QX}.
\end{rem}
\section{Appendix by Roman Bezrukavnikov, Alexander Braverman and David Kazhdan}

\subsection{Injectivity of $\bar{\phi_q}$}
 \begin{prop}
Suppose that $q$ is not a root of unity. Then map $\bar{\phi_q}$ is injective.
\end{prop}

\proof Suppose that $h\in H_q$ is such that $\phi_q(h)\in [J,J]$, we need to check that 
$h\in [H_q,H_q]$.

By density of characters (Theorem 1.3(1) in \cite{CH}) it suffices to check that 
 \begin{equation}\label{Tr0}
 Tr(\rho(h))=0
 \end{equation}
for every  irreducible representation $\rho$ of $H_q$. 

Recall the standard representation $R_{s,u,\psi}$ of $H_q$ constructed from a triple
$(s,u,\psi)$ where $s,u\in \LG$ are such that $s$ is semisimple, $u$ is unipotent, 
$sus^{-1}=u^q$ and $\psi$ is an irreducible representation of the group of components
of the simultaneous centralizer of $s$ and $u$, see \cite[Chapter 8]{CG}. 
As proved in {\em loc. cit.} (developing the original theorem of Kazhdan and Lusztig
\cite{KL}) the set of $(s,u,\rho)$ is also in bijection with irreducible representations
of $H_q$ so that the matrix of multiplicities of irreducible modules
 in $R_{s,u,\psi}$
 is upper triangular with ones on the diagonal for an appropriate
partial order on the set of triples $(s,u,\psi)$.
It follows that the classes $R_{s,u,\psi}$ form a basis
in the Grothendieck group of finite dimensional $H_q$-modules.
Thus it suffices to check \eqref{Tr0} when $\rho=R_{s,u,\psi}$ is a standard
module. 
According to \cite[Theorem 4.2]{cells4} for generic $q$ such a module is 
isomorphic to $\phi_q^*(M_{s,u,\psi})$ for a $J$-module $M_{s,u,\psi}$. 
It follows that for any 
$q$ these two modules have the same class in the Grothendieck group, 
so trace vanishing follows from the condition  $\phi_q(h)\in [J,J]$. 
\qed

\begin{rem}
Alternatively, the statement about isomorphism
of Grothendieck groups follows from \cite[Theorem 3.4]{cells3}.
While the statement about Grothendieck groups 
suffices for our purposes, in fact we have have $M_{s,u,\psi}|_{H_q}=R_{s,u,\psi}$ 
as $H_q$-modules by \cite[Corollary 2.6]{BrKa}.
\end{rem}

\begin{rem}
The definition of $J$ and $\phi_q$ (for $q$ a prime power) can be generalized: in \cite{BrKa} one finds the definition of an algebra $\bf J$ and a homomorphism
$\Phi$ from ${\bf H}=S_c^\infty (G)$, the algebra of locally constant compactly supported functions on the $p$-adic group,
such that restricting $\Phi$ to the space of Iwahori bi-invariant vectors one recovers homomorphism $\phi$ (or rather its base change
to $\Ce$).
We expect that Theorem \ref{cocenter_Thm} can be generalized to that setting:
\end{rem}

\begin{conj}\label{Phi}
The map $\Phi$ induces an isomorphism
of cocenters. 
\end{conj}

We expect that the proof of injectivity above can be generalized
using density of tempered characters \cite{Ka_cu}. 

\subsection{Conjectural description of the cocenter} \label{conj_descr}
We now present a conjecture (joint with Yakov Varshavsky) describing $C_e$, and hence a summand in the space of invariant
distributions on the $p$-adic group, in terms of the Langlands dual group.

Let $Z_e^{red}$ be the quotient of $Z_e$ by its unipotent radical. Let $\Co_e\subset (Z_e^{red})^2$ be the algebraic variety parametrizing pairs of commuting elements 
in $Z_e$. Let $\O_e$ be the ring of regular conjugation invariant functions on $\Co_e$ and $\bar{\O}_e$ be the subring of functions 
$f$ satisfying the following two conditions:

 (1) for every $x\in Z_e^{red}$ the function $f_x: y\mapsto f(x,y)$ is locally constant, i.e.
$f_x$ descends to a function $\bar{f}_x$ on $\pi_0(Z_e^{red}(x))$ where $Z_e^{red}(x)$ is the centralizer of $x$ in $Z_e^{red}$;

(2) for every $x\in Z_e^{red}$ the function $\bar{f}_x$ is a linear combination of characters of irreducible representations of 
 $\pi_0(Z_e^{red}(x))$ occurring in cohomology of the space $\BB_e^x$ of simultaneous fixed points of $e$ and $x$ on $\BB$.

\begin{conj} \label{conj}
We have a canonical isomorphism $C_e\cong \bar{\O}_e$.
\end{conj}

\begin{rem}
Supporting evidence for the conjecture is provided by the result of Aubert, Ciubotaru and Romano \cite{ACR} where the elliptic part of 
the unipotent cocenter is described in terms of $\LG$.
\end{rem}

The conjecture is the subject of a forthcoming work by Bezrukavnikov, Ciubotaru, Kazhdan and Varshavsky.

\bibliography{BDD_structure_affine_asympt_Hecke_biblio.bib}

\end{document}